\documentclass[final]{siamltexmm}
\usepackage{amsmath,amsfonts}

\def\onehalf{\tfrac{1}{2}}
\def\E{{\bf E}}
\def\neff{{n_{\rm eff}}}
\def\eq#1{Eq.~(\ref{#1})}

\def\mfe{B_F(\epsilon)}

\def\emcd{\epsilon_F(\delta)}

\newcounter{example}

\title{On the application of McDiarmid's inequality to complex
  systems\thanks{This work supported by Quantification Methods, a part
    of NNSA's Advanced Simulation and Computing Program, and from the
    Department of Energy under contract DE-AC52-06NA25396.}}

\author{Timothy C. Wallstrom\thanks{Theoretical Division, Los Alamos
    National Laboratory, Los Alamos, NM 87545, \email{tcw@lanl.gov}.}}

\begin{document}
\maketitle
\newcommand{\slugmaster}{}

\begin{abstract}
  McDiarmid's inequality has recently been proposed as a tool for
  setting margin requirements for complex systems. If $F$ is the
  bounded output of a complex system, depending on a vector of $n$
  bounded inputs, this inequality provides a bound $B_F(\epsilon)$,
  such that the probability of a deviation exceeding $B_F(\epsilon)$
  is less than $\epsilon$. I compare this bound with the absolute
  bound, based on the range of $F$. I show that when $\neff$, the
  effective number of independent variates, is small, and when
  $\epsilon$ is small, the absolute bound is smaller than
  $B_F(\epsilon)$, while also providing a smaller probability of
  exceeding the bound, i.e., zero instead of $\epsilon$. Thus, for
  $B_F(\epsilon)$ to be useful, the number of inputs must be large,
  with a small dependence on any single input, which is consistent
  with the usual guidance for application of concentration-of-measure
  results.  When the number of inputs is small, or when a small number
  of inputs account for much of the uncertainty, the absolute bounds
  will provide better results.  The use of absolute bounds is
  equivalent to the original formulation of the method of
  Quantification of Margins and Uncertainties (QMU).
\end{abstract}

\begin{keywords}
  McDiarmid inequality, concentration of measures, Quantification of
  Margins and Uncertainties (QMU), margin, uncertainty quantification,
  certification.
\end{keywords}
\begin{AMS}
60E15, 60F10, 60G50, 60G70, 65C50.
\end{AMS}

\pagestyle{myheadings}
\thispagestyle{plain}
\markboth{T.~C. WALLSTROM}{McDiarmid's inequality and complex systems}

\section{Introduction}
\label{sec:intro}

Let $F(X)=F(X_1,\ldots,X_n)$ be a bounded function of $n$ independent
real and bounded random variates $X_i$. $F$ might, for example,
represent the output of a complex system depending on $n$ random
inputs. We are interested in the fluctuations in $F$, due to
randomness in the inputs, where by fluctuation we mean deviation
$\delta$ from the mean.
McDiarmid's inequality implies that for every $\epsilon>0$, there is a
bound $\mfe$, such that the probability of a fluctuation exceeding
that bound is less than $\epsilon$:
\begin{displaymath}
  P\left(\delta > \mfe\right) < \epsilon.
\end{displaymath}
We call $\mfe$ the \textit{McDiarmid bound}.  If the inputs
induce comparable deviations  in $F$, then as $n$
increases, $\mfe$ grows roughly as $\sqrt{n}$, as we might expect from
the independence of the inputs.

There is, however, another bound for the size of these fluctuations,
which is so trivial that it tends to escape notice. Assume for
simplicity that the mean of $F$ is midway between its minimum and
maximum values. Then the largest possible fluctuation is
$\onehalf\Delta F$, where $\Delta F = F_{\max} - F_{\min}$.  When this
bound applies, it is very effective: the probability of exceeding the
bound is not just small; it is zero. This trivial bound is not useful
when $n$ is large, however. Unless we make further assumptions, it
scales as $n$, so as $n$ gets large, it rapidly loses out to $\mfe$,
which scales only as $\sqrt{n}$. 

The purpose of this paper is to observe that when $n$ is small, the
opposite is true: the McDiarmid bound is not useful, because it is
larger than the trivial bound for small $\epsilon$. Our heuristics
have assumed that each $X_i$ has roughly the same impact on
$F$, but in general, this may not be true and we need to use $\neff$,
the effective number of independent variables, as defined
below. $\neff$ can be small even when $n$ is large.  Specifically, we
will show that
\begin{equation}
  \label{eq:mcd1}
\mfe \ge  \frac{1}{2}\Delta F\,\sqrt{\frac{2}{\neff}\log \frac{1}{\epsilon}}.
\end{equation}
For McDiarmid's bound to be useful, therefore, it is necessary that
\begin{equation}
  \label{eq:mcd2}
  \neff > 2\log \frac{1}{\epsilon}.
\end{equation}
If, for example, $\epsilon=0.005$, corresponding to a two-sided
probability of exceeding the bound of 1\%, then $\neff$ must be greater
than 10.6. If $\neff < 10.6$, we should always use the absolute bound,
because it is smaller, and ensures that fluctuations exceeding the
bound happen with probability zero, rather than probability
$\epsilon$. 

McDiarmid's inequality has usually been applied in the large $n$
limit. The small $n$ properties are of interest because McDiarmid's
bound has recently been applied to the problem of deducing margin
requirements in complex systems, where $\neff$ is not always
large~\cite{Lucas:2008tr,Topcu:2011vm,Kidane:2012gd,Adams:2012gj,Owhadi:2013eo}.
In particular, it has been proposed, along with more general
concentration-of-measure~(CoM) inequalities, as a way of formalizing
the method of Quantification of Margins and
Uncertainties~(QMU)~\cite{Eardley:2005vu,Ahearne:2008wq,Wallstrom:2011ip}.
Suppose that $F$ describes a complex system, and we want to ensure
that the probability of failure~(POF) is $\epsilon$ or less. Suppose
further that we can engineer our system to function properly for any
fluctuation of size $M$ or less; $M$ is called the \textit{margin}.
If $M> \mfe$, then we know that the POF is less than $\epsilon$. The
problem of ensuring a suitably small POF is sometimes called the
\textit{certification problem}~\cite{Lucas:2008tr,Owhadi:2013eo}.

The McDiarmid inequality appears to provide a general solution to the
certification problem, because for any desired POF $\epsilon$, there
is a bound $\mfe$ that guarantees that the desired POF will not be
exceeded. 
It has never been clear, however, whether or not $\mfe$ is small
enough to be practical or useful.  If the bound is so large that the
margin cannot be made to exceed the bound, the bound will not be
useful.  One might assume that the usefulness of $\mfe$ would depend
on the details of the specific problem, and that nothing meaningful
can be said in general. The usefulness of $\mfe$ appears to be a
problem for simulation studies, which is how it has been addressed up
to this point.

The present result provides a simple necessary condition, depending
only on $\neff$ and $\epsilon$, for the McDiarmid bounds to be an
improvement on the absolute bounds.  It shows that the use of
McDiarmid bounds is only potentially an improvement for a very
specific type of problem, in which the output depends on a large
number of independent inputs, and depends only weakly on any
individual input.  For examples of such models, see
references~\cite{Lucas:2008tr} and~\cite{Topcu:2011vm}. If, on the
other hand, the number of inputs is small, or if a small number of the
inputs contribute most of the uncertainty (which leads to a small
$\neff$), then the McDiarmid bound will be larger than the trival
bound, even for moderately demanding values of $\epsilon$. We expect
many complex systems to fall into this category.

Consider, for example, the case in which $F(X)$ is a computational
model of a complex system. In such models, $n$ is generally small,
because the computational cost of exploring a high-dimensional space
is prohibitive. For example, if $n=10$ and we allot only two values to
each variate, we already require 1024 runs.  This value of ten is
smaller than the value of 10.6, cited above.  But in any case, the
relevant quantity is not $n$ but the effective number of variates
$\neff$, which accounts for the fact that some inputs have more impact
on $F$ than others, and is defined in Eq.~(\ref{eq:neff}) below.  In the
computational examples we have considered, $\neff$ is often only three
or four.

The effective number of independent variates, $\neff$, is defined in
terms of the McDiarmid diameters $D_i$:
\begin{equation}
  \label{eq:neff}
  \neff = {\left(\sum_i D_i\right)^2\over \sum_i D_i^2} =
  {1\over \sum_i f_i^2},
\end{equation}
where $f_i\equiv D_i/\sum D_i$. The McDiarmid diameters are defined in
\eq{eq:ci}, but roughly, $D_i$ represents the maximum variation in $F$
that can be caused by variations in $X_i$, with the other $X_j$ held
fixed. If the $D_i$ are all the same, then $\neff=n$. If, however, a
few of the inputs dominate the uncertainty, then $\neff$ will also be
small, regardless of how large $n$ is. For example, if $m$ of the
$D_i$ are identical and the rest are zero, then $\neff=m$. The best
estimate of the relative contribution of $X_i$ to the uncertainty is
$f_i^2$, since, as we see below~(\eq{eq:sys}), the $D_i$ add in
quadrature to form the system diameter $D_F$. Since $\neff^{-1}=\sum
f_i^2>f^2_{\max}$, we have that $\neff < 1/f_{\max}^2$, where $f_{\max}$
is the largest~$f_i$. Thus, for example, if one input contributes more
than a fourth of the uncertainty, $\neff < 4$, regardless of how large
$n$ is. In this way, $\neff$ quantifies the requirement that, in order
for the McDiarmid margin to be useful, no single input can contribute
very much to the uncertainty.

For simplicity, we have so far assumed that the mean of $F$ is midway
between its extremes, which is often approximately true. In general,
the mean may not be in the center, in which case the absolute bounds
on deviations from the mean can be written $r_\pm\Delta F$, where
$r_+$ and $r_-$ are the bounds for positive and negative deviations,
and $r_++r_-=1$. The corresponding bound on $\neff$ can be obtained
from~\eq{eq:mcd1} by rewriting the right-hand-side in terms of
$r\Delta F$, where $r=r_+$ or $r_-$. If deviations in both directions
are of interest, then we will have $r<\onehalf$ for either the
positive or negative deviations, so that the condition on $\neff$ will
be even more demanding. In some cases, deviations in only one
direction will be of interest and the corresponding $r$ may approach
one. In such cases, requirements on $\neff$ can be
reduced by a factor of up to four.

We present our main result in the next section, and then discuss 
the consequences for uncertainty quantification in the final section.


\section{Main result}
\label{sec:main}

I now define the terms and prove the main result.  Let
$X=(X_1,\ldots,X_n)$, where each $X_i$ is an independent random
variable. Let $F(x)$ be some function of $x=(x_1,\ldots,x_n)$, which
is bounded for $x_i$ in the range of $X_i$. 
The \textit{McDiarmid diameters} of $F$ are defined as
\begin{equation}
  \label{eq:ci}
  D_i \equiv \sup_{x_1,x_2,\ldots,x_i,x'_i,\ldots,x_n} 
  \left| F(x_1,\ldots,x_i,\ldots, x_n) - F(x_1,\ldots,x'_i,\ldots, x_n) \right|.
\end{equation}
The supremum and infimum, here and elsewhere, are always taken over the
values in the range of the $X_i$. The McDiarmid \textit{system diameter} is
\begin{equation}
  \label{eq:sys}
  D_F^2 = D_1^2 + D_2^2 + \cdots + D_n^2.
\end{equation}
The McDiarmid inequality~\cite{McDiarmid:1989vo} states that
\begin{equation}
  \label{eq:mcd}
  P(F(X) > \E F + \delta) \le \exp\left[-{2 \delta^2\over D_F^2}\right]\equiv \emcd.
\end{equation}
The same bound applies to $P(F(X) < \E F - \delta)$.  
The McDiarmid bound is obtained by inverting $\emcd$:
\begin{equation}
  \label{eq:meps}
  \mfe = D_F\sqrt{\frac{1}{2}\log{1\over\epsilon}  }.
\end{equation}

Using~\eq{eq:neff}, we may rewrite~\eq{eq:meps} as follows:
\begin{equation}
  \label{eq:meps-alt}
  \mfe = \left({\textstyle\sum_{i=1}^n D_i}\right)\cdot\sqrt{\frac{1}{2\neff}\log{1\over\epsilon}  }.
\end{equation}
Although this may seem like a step backwards, because we are separating
the dependence on the $D_i$ into two parts, it is useful because we
can show that $\sum D_i\ge \Delta F$. We thereby obtain our main result.

\begin{theorem} 
  Let $\mfe$ be the McDiarmid bound for probability $\epsilon$,
  and let $\Delta F =\sup_x F(x) - \inf_x F(x)$. Then
  \begin{equation}
    \label{eq:main1}
    \mfe \ge  \Delta F\,\sqrt{{1\over 2\neff}\log \frac{1}{\epsilon}}.
  \end{equation}
\end{theorem}

\begin{proof}
By~\eq{eq:meps-alt}, it suffices to prove that $\Delta F \le \sum D_i$.
But 
\begin{eqnarray*}
\Delta F &=& \sup_{x,x'} |F(x) - F(x')|\\
   &\le& \sup_{x,x'}\Big| F(x_1,x_2,\ldots,x_n) - F(x'_1,x_2,\ldots,x_n) +\cdots \\
    &&      +\; F(\ldots,x'_{i-1},x_i,x_{i+1}\ldots) - F(\ldots,x'_{i-1},x'_i,x_{i+1}\ldots)+\cdots  \\
    &&      +\; F(x'_1,x'_2, \ldots,x'_{n-1},x_n) - F(x'_1,x'_2, \ldots,x'_{n-1},x'_n)\Big| \\
   &\le& \sup_{x,x'} \sum_i \left|F(\ldots,x'_{i-1},x_i,x_{i+1}\ldots) - F(\ldots,x'_{i-1},x'_i,x_{i+1}\ldots) \right| \\
   &\le& \sum_i \sup_{x,x'} \left|F(\ldots,x'_{i-1},x_i,x_{i+1}\ldots) - F(\ldots,x'_{i-1},x'_i,x_{i+1}\ldots) \right| \\
   &\le& \sum_i D_i.
\end{eqnarray*}
\end{proof}

  The bound is an inequality because $\sum_i D_i$ may be strictly greater
  than $\Delta F$.  It is easy to show that equality holds if $F$ is
  linear in its inputs. In general, however, when there are
  interactions between the variables, the bound may be weak. For
  example, if $F = \prod_{i=1}^n x_i$, and the range of each $X_i$ is
  $[0,1]$, then $\Delta F = 1$, but $\sum D_i = n$. 

\section{Discussion}
\label{sec:disc}

Let us return to the question of how to ensure system reliability, in
view of these results. In most applications, $\neff$ will be whatever
it is, and there will be little or no freedom to change it. It may
well be small. The only question is whether we can make the margins
large enough  to ensure a suitably small POF. For many typical
systems, neither the absolute bound nor the McDiarmid bound is likely
to be of much help, because they will both require unattainable
margins.

As we have shown, if $\neff$ is small, then we are best off using the
absolute bound.  In practice, however, this bound may already be much
too large. To use this bound, we must engineer the system so that it
will work properly, even under the largest possible deviation that can
possibly occur. That is, we need to engineer the system so that
\textit{it cannot possibly fail}. In many systems, this will not be
possible.

If we do have the good fortune of a large $\neff$, then we might hope
to need a smaller margin, by virtue of the McDiarmid bounds.  However,
the concentration of measure effect kicks in very slowly, and it
requires a very large $\neff$ in order to significantly reduce the
required margin below the absolute margin. Thus, suppose that it is
only possible to engineer the system so that the margin $M$ is some
fraction $r$ of the possible range: $M=r\Delta F$. 
From~Eq.~(\ref{eq:mcd1}), we find that if $M=r\Delta F$, then
\begin{equation}
  \label{eq:mcd-r}
  \mfe \ge M\,\sqrt{\frac{1}{2r^2\,\neff}\log \frac{1}{\epsilon}},
\end{equation}
so that $\mfe$ will be greater than $M$ unless
\begin{displaymath}
  \neff > \frac{1}{2r^2}\log\frac{1}{\epsilon}.
\end{displaymath}
For $r=0.1$ and $\epsilon=0.005$, for example, we would need
$\neff>265$. 

The reason these results are so weak is that we have made no
assumptions on the input distributions, apart from their bounds.  This
is quite explicit in the case of the absolute bound, where we only use
$\Delta F$. In the McDiarmid bounds, we use only the information about
the maximum and minimum values of $F$, under variations in the $x_i$,
and no information about how likely these extreme values are. For this
reason, any bounds we obtain must be valid for the most unfavorable
distributions, which for example might have half their mass sitting on
the minimum, and half on the maximum. Such configurations are often
known to be highly unlikely or impossible, so these analyses may not
be using important information. For an approach that uses such
information, see Ref.~\cite{Wallstrom:2011ip}.

Finally, it is interesting to note that the use of absolute bounds is
essentially QMU, in its original
formulation~\cite{Goodwin:2006ud}. $F$ is the \textit{metric}, the
deviations are measured with respect to the \textit{design point},
$\onehalf (F_{\max} + F_{\min})$, $\onehalf \Delta F$ is the
uncertainty $U$, and $M/U$, which should be greater than one, is the
\textit{confidence ratio}.  (For simplicity, I again assume symmetry
for positive and negative deviations, but these conditions can
trivially be made asymmetric.) The condition that $M/U>1$ is identical
to our condition that $M$ be greater than the bound on the
fluctuations.  Formally, QMU provides a POF of zero.  Thus, our result
can be stated concisely as follows: For small $\neff$ and $\epsilon$,
QMU provides a POF of zero with smaller margins than those required by
the McDiarmid bounds for a POF of $\epsilon$.  QMU was designed for
extremely reliable systems, as a way of certifying that if the margins
of key quantities were sufficiently large, the system could not
possibly fail. In such systems, it may be feasible to use the absolute
bound.

\section*{Acknowledgments}

This work grew out of an attempt to understand simulation results of
Fran\c cois Hemez and Christopher Stull. I thank them for bringing
this interesting problem to my attention.


\end{document}